\documentclass[11pt,fleqn]{article}

\usepackage{amsmath,amssymb,amsthm}
\usepackage{geometry}
\usepackage[pdfpagemode=UseNone,pdfstartview=FitH]{hyperref}

\newcommand{\abs}[1]{\left|#1\right|}
\newcommand{\bdry}[1]{\partial #1}
\newcommand{\bgset}[1]{\big\{#1\big\}}
\newcommand{\A}{{\cal A}}
\newcommand{\F}{{\cal F}}
\newcommand{\comp}{\circ}
\newcommand{\dist}[2]{\text{dist}\, (#1,#2)}
\newcommand{\ds}[1]{\displaystyle #1}
\newcommand{\eps}{\varepsilon}
\newcommand{\id}[1][]{id_{\, #1}}
\newcommand{\incl}{\subset}
\newcommand{\M}{{\cal M}}
\newcommand{\N}{\mathbb N}
\newcommand{\norm}[2][]{\left\|#2\right\|_{#1}}
\renewcommand{\o}{\text{o}}
\newcommand{\PS}[1]{$(\text{PS})_{#1}$}
\newcommand{\pnorm}[2][]{\if #1'' \left|#2\right|_p \else \left|#2\right|_{#1} \fi}
\newcommand{\QED}{\mbox{\qedhere}}
\newcommand{\R}{\mathbb R}
\newcommand{\RP}{\R \text{P}}
\newcommand{\restr}[2]{\left.#1\right|_{#2}}
\newcommand{\seq}[1]{\left(#1\right)}
\newcommand{\set}[1]{\left\{#1\right\}}
\newcommand{\vol}[1]{\left|#1\right|}
\newcommand{\wto}{\rightharpoonup}
\newcommand{\Z}{\mathbb Z}

\newenvironment{enumroman}{\begin{enumerate}
\renewcommand{\theenumi}{$(\roman{enumi})$}
\renewcommand{\labelenumi}{$(\roman{enumi})$}}{\end{enumerate}}

\newenvironment{properties}[1]{\begin{enumerate}
\renewcommand{\theenumi}{$(#1_\arabic{enumi})$}
\renewcommand{\labelenumi}{$(#1_\arabic{enumi})$}}{\end{enumerate}}

\newtheorem{corollary}{Corollary}[section]
\newtheorem{lemma}[corollary]{Lemma}
\newtheorem{proposition}[corollary]{Proposition}
\newtheorem{theorem}[corollary]{Theorem}

\theoremstyle{remark}
\newtheorem{remark}[corollary]{Remark}

\numberwithin{equation}{section}

\title{\bf Bifurcation and multiplicity results for critical fractional $p$-Laplacian problems}
\author{\bf Kanishka Perera\\
Department of Mathematical Sciences\\
Florida Institute of Technology\\
Melbourne, FL 32901, USA\\
[\bigskipamount]
\bf Marco Squassina\thanks{The second-named author was supported by 2009 MIUR project: ``Variational and Topological Methods in the Study of Nonlinear Phenomena''.}\\
Dipartimento di Informatica\\
Universit\`a degli Studi di Verona\\
37134 Verona, Italy\\
[\bigskipamount]
\bf Yang Yang\thanks{This work was completed while the third-named author was visiting the Department of Mathematical Sciences at the Florida Institute of Technology, and she is grateful for the kind hospitality of the department. Project supported by NSFC-Tian Yuan Special Foundation (No. 11226116), Natural Science Foundation of Jiangsu Province of China for Young Scholars (No. BK2012109), and the China Scholarship Council (No. 201208320435).}\\
School of Science\\
Jiangnan University\\
Wuxi, 214122, China}
\date{}

\begin{document}

\maketitle

\begin{abstract}
We prove a bifurcation and multiplicity result for a critical fractional $p$-Laplacian problem that is the analog of the Br{\'e}zis-Nirenberg problem for the nonlocal quasilinear case. This extends a result in the literature for the semilinear case $p = 2$ to all $p \in (1,\infty)$, in particular, it gives a new existence result. When $p \ne 2$, the nonlinear operator $(- \Delta)_p^s,\, s \in (0,1)$ has no linear eigenspaces, so our extension is nontrivial and requires a new abstract critical point theorem that is not based on linear subspaces. We prove a new abstract result based on a pseudo-index related to the $\Z_2$-cohomological index that is applicable here.
\end{abstract}

\section{Introduction and main results}

For $p \in (1,\infty)$, $s \in (0,1)$, and $N > sp$, the fractional $p$-Laplacian $(- \Delta)_p^s$ is the nonlinear nonlocal operator defined on smooth functions by
\[
(- \Delta)_p^s\, u(x) = 2 \lim_{\eps \searrow 0} \int_{\R^N \setminus B_\eps(x)} \frac{|u(x) - u(y)|^{p-2}\, (u(x) - u(y))}{|x - y|^{N+sp}}\, dy, \quad x \in \R^N.
\]
This definition is consistent, up to a normalization constant depending on $N$ and $s$, with the usual definition of the linear fractional Laplacian operator $(- \Delta)^s$ when $p = 2$. 
 There is currently a rapidly growing literature on problems involving these nonlocal operators. In particular, fractional $p$-eigenvalue problems have been studied in
 Brasco and Parini \cite{eig1}, Brasco, Parini and Squassina \cite{eig2}, Franzina and Palatucci \cite{FrPa},
 Iannizzotto and Squassina \cite{MR3245079} and in Lindgren and Lindqvist \cite{MR3148135}. 
 Regularity of solutions was obtained in Di Castro, Kuusi and  Palatucci \cite{DiKuPa,reg2}, Iannizzotto, Mosconi and  Squassina \cite{reg4},
 Kuusi, Mingione and Sire \cite{reg1} and Lindgren \cite{reg3}.
 Existence via Morse theory was investigated in Iannizzotto et al.\! \cite{IaLiPeSq}. We refer to Caffarelli \cite{Ca} for the motivations that have lead to their study.
Let $\Omega$ be a bounded domain in $\R^N$ with Lipschitz boundary. We consider the problem
\begin{equation} \label{1}
\left\{\begin{aligned}
(- \Delta)_p^s\, u & = \lambda\, |u|^{p-2}\, u + |u|^{p_s^\ast - 2}\, u && \text{in } \Omega\\[10pt]
u & = 0 && \text{in } \R^N \setminus \Omega,
\end{aligned}\right.
\end{equation}
where $p_s^\ast = Np/(N - sp)$ is the fractional critical Sobolev exponent. Let us recall the weak formulation of problem \eqref{1}. Let
\[
[u]_{s,p} = \left(\int_{\R^{2N}} \frac{|u(x) - u(y)|^p}{|x - y|^{N+sp}}\, dx dy\right)^{1/p}
\]
be the Gagliardo seminorm of the measurable function $u : \R^N \to \R$, and let
\[
W^{s,p}(\R^N) = \set{u \in L^p(\R^N) : [u]_{s,p} < \infty}
\]
be the fractional Sobolev space endowed with the norm
\[
\norm[s,p]{u} = \big(\pnorm[p]{u}^p + [u]_{s,p}^p\big)^{1/p},
\]
where $\pnorm[p]{\cdot}$ is the norm in $L^p(\R^N)$ (see Di Nezza et al.\! \cite{MR2944369} for details). We work in the closed linear subspace
\[
X_p^s(\Omega) = \set{u \in W^{s,p}(\R^N) : u = 0 \text{ a.e.\! in } \R^N \setminus \Omega},
\]
equivalently renormed by setting $\norm{\cdot} = [\cdot]_{s,p}$, which is a uniformly convex Banach space. By \cite[Theorems 6.5 \& 7.1]{MR2944369}, the imbedding $X_p^s(\Omega) \hookrightarrow L^r(\Omega)$ is continuous for $r \in [1,p_s^\ast]$ and compact for $r \in [1,p_s^\ast)$. We let
\begin{equation} \label{7}
S_{s,p} = \inf_{u \in X_p^s(\Omega) \setminus \set{0}}\, \frac{\norm{u}^p}{\pnorm[p_s^\ast]{u}^p}
\end{equation}
denote the best imbedding constant when $r = p_s^\ast$. A function $u \in X_p^s(\Omega)$ is a weak solution of problem \eqref{1} if
\begin{multline} \label{10}
\int_{\R^{2N}} \frac{|u(x) - u(y)|^{p-2}\, (u(x) - u(y))\, (v(x) - v(y))}{|x - y|^{N+sp}}\, dx dy\\[5pt]
= \lambda \int_\Omega |u|^{p-2}\, uv\, dx + \int_\Omega |u|^{p_s^\ast - 2}\, uv\, dx \quad \forall v \in X_p^s(\Omega).
\end{multline}

In the semilinear case $p = 2$, problem \eqref{1} reduces to the critical fractional Laplacian problem
\begin{equation} \label{2}
\left\{\begin{aligned}
(- \Delta)^s\, u & = \lambda u + |u|^{2_s^\ast - 2}\, u && \text{in } \Omega\\[10pt]
u & = 0 && \text{in } \R^N \setminus \Omega,
\end{aligned}\right.
\end{equation}
where $2_s^\ast = 2N/(N - 2s)$. This nonlocal problem generalizes the well-known Br{\'e}zis\nobreakdash-Nirenberg problem, which has been extensively studied beginning with the seminal paper \cite{MR709644} (see, e.g., \cite{MR779872,MR831041,MR829403,MR1009077,MR1083144,MR1154480,MR1306583,MR1473856,MR1441856,MR1491613,MR1695021,MR1784441,MR1961520} and references therein). Consequently, many results known in the local case $s = 1$ have been extended to problem \eqref{2} (see, e.g., \cite{MR3089742,Se,MR3060890,SeVa1,MR2911424,MR2819627,FiscellaBisciServadei}). In particular, Fiscella et al.\! \cite{FiscellaBisciServadei} have recently obtained the following bifurcation and multiplicity result, extending a well-known result of Cerami et al.\! \cite{MR779872} in the local case. Let $0 < \lambda_1 < \lambda_2 \le \lambda_3 \le \cdots \to + \infty$ be the eigenvalues of the problem
\[
\left\{\begin{aligned}
(- \Delta)^s\, u & = \lambda u && \text{in } \Omega\\[10pt]
u & = 0 && \text{in } \R^N \setminus \Omega,
\end{aligned}\right.
\]
repeated according to multiplicity, and let $\vol{\cdot}$ denote the Lebesgue measure in $\R^N$. If $\lambda_k \le \lambda < \lambda_{k+1}$ and
\[
\lambda > \lambda_{k+1} - \frac{S_{s,2}}{\vol{\Omega}^{2s/N}},
\]
and $m$ denotes the multiplicity of $\lambda_{k+1}$, then problem \eqref{2} has $m$ distinct pairs of nontrivial solutions $\pm\, u^\lambda_j,\, j = 1,\dots,m$ such that $u^\lambda_j \to 0$ as $\lambda \nearrow \lambda_{k+1}$ (see \cite[Theorem 1]{FiscellaBisciServadei}).

In the present paper we extend the above bifurcation and multiplicity result to the quasilinear nonlocal problem \eqref{1}. This extension is quite nontrivial. Indeed, the linking argument based on eigenspaces of $(- \Delta)^s$ in \cite{FiscellaBisciServadei} does not work when $p \ne 2$ since the nonlinear operator $(- \Delta)_p^s$ does not have linear eigenspaces. We will use a more general construction based on sublevel sets as in Perera and Szulkin \cite{MR2153141} (see also Perera et al.\! \cite[Proposition 3.23]{MR2640827}). Moreover, the standard sequence of variational eigenvalues of $(- \Delta)_p^s$ based on the genus does not give enough information about the structure of the sublevel sets to carry out this linking construction. Therefore we will use a different sequence of eigenvalues introduced in Iannizzotto et al.\! \cite{IaLiPeSq} that is based on the $\Z_2$-cohomological index of Fadell and Rabinowitz \cite{MR57:17677}, which is defined as follows. Let $W$ be a Banach space and let $\A$ denote the class of symmetric subsets of $W \setminus \set{0}$. For $A \in \A$, let $\overline{A} = A/\Z_2$ be the quotient space of $A$ with each $u$ and $-u$ identified, let $f : \overline{A} \to \RP^\infty$ be the classifying map of $\overline{A}$, and let $f^\ast : H^\ast(\RP^\infty) \to H^\ast(\overline{A})$ be the induced homomorphism of the Alexander-Spanier cohomology rings. The cohomological index of $A$ is defined by
\[
i(A) = \begin{cases}
\sup \set{m \ge 1 : f^\ast(\omega^{m-1}) \ne 0}, & A \ne \emptyset\\[5pt]
0, & A = \emptyset,
\end{cases}
\]
where $\omega \in H^1(\RP^\infty)$ is the generator of the polynomial ring $H^\ast(\RP^\infty) = \Z_2[\omega]$. For example, the classifying map of the unit sphere $S^{m-1}$ in $\R^m,\, m \ge 1$ is the inclusion $\RP^{m-1} \incl \RP^\infty$, which induces isomorphisms on $H^q$ for $q \le m - 1$, so $i(S^{m-1}) = m$.

The Dirichlet spectrum of $(- \Delta)_p^s$ in $\Omega$ consists of those $\lambda \in \R$ for which the problem
\begin{equation} \label{3}
\left\{\begin{aligned}
(- \Delta)_p^s\, u & = \lambda\, |u|^{p-2}\, u && \text{in } \Omega\\[10pt]
u & = 0 && \text{in } \R^N \setminus \Omega
\end{aligned}\right.
\end{equation}
has a nontrivial solution. Although a complete description of the spectrum is not known when $p \ne 2$, we can define an increasing and unbounded sequence of variational eigenvalues via a suitable minimax scheme. The standard scheme based on the genus does not give the index information necessary for our purposes here, so we will use the following scheme based on the cohomological index as in Iannizzotto et al.\! \cite{IaLiPeSq} (see also Perera \cite{MR1998432}). Let
\[
\Psi(u) = \frac{1}{\pnorm[p]{u}^p}, \quad u \in \M = \set{u \in X_p^s(\Omega) : \norm{u} = 1}.
\]
Then eigenvalues of problem \eqref{3} coincide with critical values of $\Psi$. We use the standard notation
\[
\Psi^a = \set{u \in \M : \Psi(u) \le a}, \quad \Psi_a = \set{u \in \M : \Psi(u) \ge a}, \quad a \in \R
\]
for the sublevel sets and superlevel sets, respectively. Let $\F$ denote the class of symmetric subsets of $\M$, and set
\[
\lambda_k := \inf_{M \in \F,\; i(M) \ge k}\, \sup_{u \in M}\, \Psi(u), \quad k \in \N.
\]
Then $0 < \lambda_1 < \lambda_2 \le \lambda_3 \le \cdots \to + \infty$ is a sequence of eigenvalues of problem \eqref{3}, and
\begin{equation} \label{4}
\lambda_k < \lambda_{k+1} \implies i(\Psi^{\lambda_k}) = i(\M \setminus \Psi_{\lambda_{k+1}}) = k
\end{equation}
(see Iannizzotto et al.\! \cite[Proposition 2.4]{IaLiPeSq}). The asymptotic behavior of these eigenvalues was recently studied in Iannizzotto and Squassina \cite{MR3245079}. Making essential use of the index information in \eqref{4}, we will prove the following theorem.

\begin{theorem} \label{Theorem 1}
\begin{enumroman}
\item \label{Theorem 1.i} If
\[
\lambda_1 - \frac{S_{s,p}}{\vol{\Omega}^{sp/N}} < \lambda < \lambda_1,
\]
then problem \eqref{1} has a pair of nontrivial solutions $\pm\, u^\lambda$ such that $u^\lambda \to 0$ as $\lambda \nearrow \lambda_1$.

\item \label{Theorem 1.ii} If $\lambda_k \le \lambda < \lambda_{k+1} = \cdots = \lambda_{k+m} < \lambda_{k+m+1}$ for some $k, m \in \N$ and
\begin{equation} \label{5}
\lambda > \lambda_{k+1} - \frac{S_{s,p}}{\vol{\Omega}^{sp/N}},
\end{equation}
then problem \eqref{1} has $m$ distinct pairs of nontrivial solutions $\pm\, u^\lambda_j,\, j = 1,\dots,m$ such that $u^\lambda_j \to 0$ as $\lambda \nearrow \lambda_{k+1}$.
\end{enumroman}
\end{theorem}

In particular, we have the following existence result.

\begin{corollary}
Problem \eqref{1} has a nontrivial solution for all $\lambda \in \ds{\bigcup_{k=1}^\infty} \big(\lambda_k - S_{s,p}/\vol{\Omega}^{sp/N},\lambda_k\big)$.
\end{corollary}

We note that $\lambda_1 \ge S_{s,p}/\vol{\Omega}^{sp/N}$. Indeed, if $\varphi_1$ is an eigenfunction associated with $\lambda_1$,
\[
\lambda_1 = \frac{\norm{\varphi_1}^p}{\pnorm[p]{\varphi_1}^p} \ge \frac{S_{s,p} \pnorm[p_s^\ast]{\varphi_1}^p}{\pnorm[p]{\varphi_1}^p} \ge \frac{S_{s,p}}{\vol{\Omega}^{sp/N}}
\]
by the H\"{o}lder inequality.

\begin{remark}
Analogous results for the corresponding local problem driven by the $p$-Laplacian operator were recently obtained in Perera
et al.\ \cite{PeSqYa1}.\ In this work new difficulties need to be handled due to the non-locality of the problem, in particular the detection
of a (lower) range of validity of the Palais-Smale condition, consistent with the one known in the local case $s=1$, see Proposition~\ref{Proposition 4}.
\end{remark}

\section{An abstract critical point theorem}

The abstract result of Bartolo et al. \cite{MR713209} used in Cerami et al.\! \cite{MR779872} and Fiscella et al.\! \cite{FiscellaBisciServadei} is based on linear subspaces and therefore cannot be used to prove our Theorem \ref{Theorem 1}. In this section we prove a more general critical point theorem based on a pseudo-index related to the cohomological index that is applicable here (see also Perera et al.\! \cite[Proposition 3.44]{MR2640827}).

Let $W$ be a Banach space and let $\A$ denote the class of symmetric subsets of $W \setminus \set{0}$. The following proposition summarizes the basic properties of the cohomological index.

\begin{proposition}[Fadell-Rabinowitz \cite{MR57:17677}] \label{Proposition 2}
The index $i : \A \to \N \cup \set{0,\infty}$ has the following properties:
\begin{properties}{i}
\item Definiteness: $i(A) = 0$ if and only if $A = \emptyset$;
\item \label{i2} Monotonicity: If there is an odd continuous map from $A$ to $B$ (in particular, if $A \subset B$), then $i(A) \le i(B)$. Thus, equality holds when the map is an odd homeomorphism;
\item Dimension: $i(A) \le \dim W$;
\item Continuity: If $A$ is closed, then there is a closed neighborhood $N \in \A$ of $A$ such that $i(N) = i(A)$. When $A$ is compact, $N$ may be chosen to be a $\delta$-neighborhood $N_\delta(A) = \set{u \in W : \dist{u}{A} \le \delta}$;
\item Subadditivity: If $A$ and $B$ are closed, then $i(A \cup B) \le i(A) + i(B)$;
\item \label{i6} Stability: If $SA$ is the suspension of $A \ne \emptyset$, obtained as the quotient space of $A \times [-1,1]$ with $A \times \set{1}$ and $A \times \set{-1}$ collapsed to different points, then $i(SA) = i(A) + 1$;
\item \label{i7} Piercing property: If $A$, $A_0$ and $A_1$ are closed, and $\varphi : A \times [0,1] \to A_0 \cup A_1$ is a continuous map such that $\varphi(-u,t) = - \varphi(u,t)$ for all $(u,t) \in A \times [0,1]$, $\varphi(A \times [0,1])$ is closed, $\varphi(A \times \set{0}) \subset A_0$ and $\varphi(A \times \set{1}) \subset A_1$, then $i(\varphi(A \times [0,1]) \cap A_0 \cap A_1) \ge i(A)$;
\item Neighborhood of zero: If $U$ is a bounded closed symmetric neighborhood of $0$, then $i(\bdry{U}) = \dim W$.
\end{properties}
\end{proposition}

Let $\Phi$ be an even $C^1$-functional defined on $W$, and recall that $\Phi$ satisfies the Palais\nobreakdash-Smale compactness condition at the level $c \in \R$, or \PS{c} for short, if every sequence $\seq{u_j} \subset W$ such that $\Phi(u_j) \to c$ and $\Phi'(u_j) \to 0$ has a convergent subsequence. Let $\A^\ast$ denote the class of symmetric subsets of $W$, let $r > 0$, let $S_r = \set{u \in W : \norm{u} = r}$, let $0 < b \le + \infty$, and let $\Gamma$ denote the group of odd homeomorphisms of $W$ that are the identity outside $\Phi^{-1}(0,b)$. The pseudo-index of $M \in \A^\ast$ related to $i$, $S_r$ and $\Gamma$ is defined by
\[
i^\ast(M) = \min_{\gamma \in \Gamma}\, i(\gamma(M) \cap S_r)
\]
(see Benci \cite{MR84c:58014}). The following critical point theorem generalizes Bartolo et al. \cite[Theorem 2.4]{MR713209}.

\begin{theorem} \label{Theorem 3}
Let $A_0,\, B_0$ be symmetric subsets of $S_1$ such that $A_0$ is compact, $B_0$ is closed, and
\[
i(A_0) \ge k + m, \qquad i(S_1 \setminus B_0) \le k
\]
for some integers $k \ge 0$ and $m \ge 1$. Assume that there exists $R > r$ such that
\[
\sup \Phi(A) \le 0 < \inf \Phi(B), \qquad \sup \Phi(X) < b,
\]
where $A = \set{Ru : u \in A_0}$, $B = \set{ru : u \in B_0}$, and $X = \set{tu : u \in A,\, 0 \le t \le 1}$. For $j = k + 1,\dots,k + m$, let
\[
\A_j^\ast = \set{M \in \A^\ast : M \text{ is compact and } i^\ast(M) \ge j}
\]
and set
\[
c_j^\ast := \inf_{M \in \A_j^\ast}\, \max_{u \in M}\, \Phi(u).
\]
Then
\[
\inf \Phi(B) \le c_{k+1}^\ast \le \dotsb \le c_{k+m}^\ast \le \sup \Phi(X),
\]
in particular, $0 < c_j^\ast < b$. If, in addition, $\Phi$ satisfies the {\em \PS{c}} condition for all $c \in (0,b)$, then each $c_j^\ast$ is a critical value of $\Phi$ and there are $m$ distinct pairs of associated critical points.
\end{theorem}

\begin{proof}
If $M \in \A_{k+1}^\ast$,
\[
i(S_r \setminus B) = i(S_1 \setminus B_0) \le k < k + 1 \le i^\ast(M) \le i(M \cap S_r)
\]
since $\id[W] \in \Gamma$. Hence $M$ intersects $B$ by \ref{i2} of Proposition \ref{Proposition 2}. It follows that $c_{k+1}^\ast \ge \inf \Phi(B)$. If $\gamma \in \Gamma$, consider the continuous map
\[
\varphi : A \times [0,1] \to W, \quad \varphi(u,t) = \gamma(tu).
\]
We have $\varphi(A \times [0,1]) = \gamma(X)$, which is compact. Since $\gamma$ is odd, $\varphi(-u,t) = - \varphi(u,t)$ for all $(u,t) \in A \times [0,1]$ and $\varphi(A \times \set{0}) = \set{\gamma(0)} = \set{0}$. Since $\Phi \le 0$ on $A$, $\restr{\gamma}{A} = \id[A]$ and hence $\varphi(A \times \set{1}) = A$. Applying \ref{i7} with $\widetilde{A}_0 = \set{u \in W : \norm{u} \le r}$ and $\widetilde{A}_1 = \set{u \in W : \norm{u} \ge r}$ gives
\[
i(\gamma(X) \cap S_r) = i(\varphi(A \times [0,1]) \cap \widetilde{A}_0 \cap \widetilde{A}_1) \ge i(A) = i(A_0) \ge k + m.
\]
It follows that $i^\ast(X) \ge k + m$. So $X \in \A_{k+m}^\ast$ and hence $c_{k+m}^\ast \le \sup \Phi(X)$. The rest now follows from standard arguments in critical point theory (see, e.g., Perera et al.\! \cite{MR2640827}).
\end{proof}

\begin{remark}
Constructions similar to the one in the proof of Theorem \ref{Theorem 3} have been used in Fadell and Rabinowitz \cite{MR57:17677} to prove bifurcation results for Hamiltonian systems, and in Perera and Szulkin \cite{MR2153141} to obtain nontrivial solutions of $p$-Laplacian problems with nonlinearities that interact with the spectrum. See also Perera et al.\! \cite[Proposition 3.44]{MR2640827}.
\end{remark}

\section{Proof of Theorem \ref{Theorem 1}}

Weak solutions of problem \eqref{1} coincide with critical points of the $C^1$-functional
\[
I_\lambda(u) = \frac{1}{p} \norm{u}^p - \frac{\lambda}{p} \pnorm[p]{u}^p - \frac{1}{p_s^\ast} \pnorm[p_s^\ast]{u}^{p_s^\ast}, \quad u \in X_p^s(\Omega).
\]
We have the following compactness result, which is well-known in the local case $s = 1$.

\begin{proposition} \label{Proposition 4}
For any $\lambda \in \R$, $I_\lambda$ satisfies the {\em \PS{c}} condition for all $c < \frac{s}{N}\, S_{s,p}^{N/sp}$.
\end{proposition}

First we prove a lemma.

\begin{lemma} \label{Lemma 5}
If $\seq{u_j}$ is bounded in $X_p^s(\Omega)$ and $u_j \to u$ a.e.\! in $\Omega$, then
\[
\norm{u_j}^p = \norm{u_j - u}^p + \norm{u}^p + \o(1) \quad \text{as } j \to \infty.
\]
\end{lemma}

\begin{proof}
Defining $\omega_j : \R^{2N} \to \R_+$ by
\[
\omega_j(x,y) = \abs{\frac{|u_j(x) - u_j(y)|^p}{|x - y|^{N+sp}} - \frac{|(u_j(x) - u(x)) - (u_j(y) - u(y))|^p}{|x - y|^{N+sp}} - \frac{|u(x) - u(y)|^p}{|x - y|^{N+sp}}},
\]
we will show that
\begin{equation} \label{6}
\lim_{j \to \infty}\, \int_{\R^{2N}} \omega_j(x,y)\, dx dy = 0.
\end{equation}
Given $\eps > 0$, there exists $C_\eps > 0$ such that
\[
\big||a + b|^p - |a|^p\big| \le \eps\, |a|^p + C_\eps\, |b|^p \quad \forall a, b \in \R,
\]
and taking $a = (u_j(x) - u_j(y)) - (u(x) - u(y))$ and $b = u(x) - u(y)$ gives
\[
\omega_j(x,y) \le \eps\, \frac{|(u_j(x) - u_j(y)) - (u(x) - u(y))|^p}{|x - y|^{N+sp}} + C_\eps\, \frac{|u(x) - u(y)|^p}{|x - y|^{N+sp}}.
\]
Consequently, defining $\omega_j^\eps : \R^{2N} \to \R_+$ by
\[
\omega_j^\eps(x,y) = \left(\omega_j(x,y) - \eps\, \frac{|(u_j(x) - u_j(y)) - (u(x) - u(y))|^p}{|x - y|^{N+sp}}\right)^+,
\]
we have
\[
\omega_j^\eps(x,y) \le C_\eps\, \frac{|u(x) - u(y)|^p}{|x - y|^{N+sp}} \in L^1(\R^{2N}).
\]
Since $u_j \to u$ a.e.\! in $\R^N$, $\omega_j^\eps \to 0$ a.e.\! in $\R^{2N}$, so the dominated convergence theorem now implies
\[
\lim_{j \to \infty}\, \int_{\R^{2N}} \omega_j^\eps(x,y)\, dx dy = 0.
\]
Then
\[
\limsup_{j \to \infty}\, \int_{\R^{2N}} \omega_j(x,y)\, dx dy \le \eps\, \limsup_{j \to \infty}\, \int_{\R^{2N}} \frac{|(u_j(x) - u_j(y)) - (u(x) - u(y))|^p}{|x - y|^{N+sp}}\, dx dy,
\]
and \eqref{6} follows since $\eps > 0$ is arbitrary and $u_j$ is bounded in $X_p^s(\Omega)$.
\end{proof}

\begin{proof}[Proof of Proposition \ref{Proposition 4}]
Let $c < \frac{s}{N}\, S_{s,p}^{N/sp}$ and let $\seq{u_j}$ be a sequence in $X_p^s(\Omega)$ such that
\begin{gather}
\label{8} I_\lambda(u_j) = \frac{1}{p} \norm{u_j}^p - \frac{\lambda}{p} \pnorm[p]{u_j}^p - \frac{1}{p_s^\ast} \pnorm[p_s^\ast]{u_j}^{p_s^\ast} = c + \o(1),\\[10pt]
\label{9} \begin{split}
I_\lambda'(u_j)\, v = & \int_{\R^{2N}} \frac{|u_j(x) - u_j(y)|^{p-2}\, (u_j(x) - u_j(y))\, (v(x) - v(y))}{|x - y|^{N+sp}}\, dx dy\\[5pt]
& - \lambda \int_\Omega |u_j|^{p-2}\, u_j\, v\, dx - \int_\Omega |u_j|^{p_s^\ast - 2}\, u_j\, v\, dx = \o(\norm{v}) \quad \forall v \in X_p^s(\Omega)
\end{split}
\end{gather}
as $j \to \infty$. Then
\[
\frac{s}{N} \pnorm[p_s^\ast]{u_j}^{p_s^\ast} = I_\lambda(u_j) - \frac{1}{p}\, I_\lambda'(u_j)\, u_j = \o(\norm{u_j}) + \O(1),
\]
which together with \eqref{8} and the H\"{o}lder inequality shows that $\seq{u_j}$ is bounded in $X_p^s(\Omega)$. So a renamed subsequence of $\seq{u_j}$ converges to some $u$ weakly in $X_p^s(\Omega)$, strongly in $L^r(\Omega)$ for all $r \in [1,p_s^\ast)$, and a.e.\! in $\Omega$ (see Di Nezza et al.\! \cite[Corollary 7.2]{MR2944369}). Denoting by $p' = p/(p - 1)$ the H\"{o}lder conjugate of $p$, $|u_j(x) - u_j(y)|^{p-2}\, (u_j(x) - u_j(y))/|x - y|^{(N+sp)/p'}$ is bounded in $L^{p'}(\R^{2N})$ and converges to $|u(x) - u(y)|^{p-2}\, (u(x) - u(y))/|x - y|^{(N+sp)/p'}$ a.e.\! in $\R^{2N}$, and $(v(x) - v(y))/|x - y|^{(N+sp)/p} \in L^p(\R^{2N})$, so the first integral in \eqref{9} converges to
\[
\int_{\R^{2N}} \frac{|u(x) - u(y)|^{p-2}\, (u(x) - u(y))\, (v(x) - v(y))}{|x - y|^{N+sp}}\, dx dy
\]
for a further subsequence. Moreover,
\[
\int_\Omega |u_j|^{p-2}\, u_j\, v\, dx \to \int_\Omega |u|^{p-2}\, uv\, dx,
\]
and
\[
\int_\Omega |u_j|^{p_s^\ast - 2}\, u_j\, v\, dx \to \int_\Omega |u|^{p_s^\ast - 2}\, uv\, dx
\]
since $|u_j|^{p_s^\ast - 2}\, u_j \wto |u|^{p_s^\ast - 2}\, u$ in $L^{(p_s^\ast)'}(\Omega)$. So passing to the limit in \eqref{9} shows that $u \in X_p^s(\Omega)$ is a weak solution of \eqref{1}, i.e., \eqref{10} holds.

Setting $\widetilde{u}_j = u_j - u$, we will show that $\widetilde{u}_j \to 0$ in $X_p^s(\Omega)$. We have
\begin{equation} \label{11}
\norm{\widetilde{u}_j}^p = \norm{u_j}^p - \norm{u}^p + \o(1)
\end{equation}
by Lemma \ref{Lemma 5}, and
\begin{equation}
\pnorm[p_s^\ast]{\widetilde{u}_j}^{p_s^\ast} = \pnorm[p_s^\ast]{u_j}^{p_s^\ast} - \pnorm[p_s^\ast]{u}^{p_s^\ast} + \o(1)
\end{equation}
by the Br{\'e}zis-Lieb lemma \cite[Theorem 1]{MR699419}. Taking $v = u_j$ in \eqref{9} gives
\begin{equation} \label{14}
\norm{u_j}^p = \lambda \pnorm[p]{u}^p + \pnorm[p_s^\ast]{u_j}^{p_s^\ast} + \o(1)
\end{equation}
since $\seq{u_j}$ is bounded in $X_p^s(\Omega)$ and converges to $u$ in $L^p(\Omega)$, and testing \eqref{10} with $v = u$ gives
\begin{equation} \label{12}
\norm{u}^p = \lambda \pnorm[p]{u}^p + \pnorm[p_s^\ast]{u}^{p_s^\ast}.
\end{equation}
It follows from \eqref{11}--\eqref{12} and \eqref{7} that
\[
\norm{\widetilde{u}_j}^p = \pnorm[p_s^\ast]{\widetilde{u}_j}^{p_s^\ast} + \o(1) \le \frac{\norm{\widetilde{u}_j}^{p_s^\ast}}{S_{s,p}^{p_s^\ast/p}} + \o(1),
\]
so
\begin{equation} \label{13}
\norm{\widetilde{u}_j}^p \big(S_{s,p}^{p_s^\ast/p} - \norm{\widetilde{u}_j}^{p_s^\ast - p}\big) \le \o(1).
\end{equation}
On the other hand,
\[
\begin{aligned}
c & = \frac{1}{p} \norm{u_j}^p - \frac{\lambda}{p} \pnorm[p]{u}^p - \frac{1}{p_s^\ast} \norm[p_s^\ast]{u_j}^{p_s^\ast} + \o(1) && \text{by \eqref{8}}\\[10pt]
& = \frac{s}{N}\, \big(\norm{u_j}^p - \lambda \pnorm[p]{u}^p\big) + \o(1) && \text{by \eqref{14}}\\[10pt]
& = \frac{s}{N}\, \big(\norm{\widetilde{u}_j}^p + \norm{u}^p - \lambda \pnorm[p]{u}^p\big) + \o(1) && \text{by \eqref{11}}\\[10pt]
& = \frac{s}{N}\, \big(\norm{\widetilde{u}_j}^p + \pnorm[p_s^\ast]{u}^{p_s^\ast}\big) + \o(1) && \text{by \eqref{12}}\\[10pt]
& \ge \frac{s}{N} \norm{\widetilde{u}_j}^p + \o(1),
\end{aligned}
\]
so
\begin{equation} \label{15}
\limsup_{j \to \infty}\, \norm{\widetilde{u}_j}^p \le \frac{Nc}{s} < S_{s,p}^{N/sp}.
\end{equation}
It follows from \eqref{13} and \eqref{15} that $\norm{\widetilde{u}_j} \to 0$.
\end{proof}

If $\lambda_{k+m} < \lambda_{k+m+1}$, then $i(\Psi^{\lambda_{k+m}}) = k + m$ by \eqref{4}. In order to apply Theorem \ref{Theorem 3} to the functional $I_\lambda$ to prove Theorem \ref{Theorem 1}, we will construct a compact symmetric subset $A_0$ of $\Psi^{\lambda_{k+m}}$ with the same index. As noted in Iannizzotto et al.\! \cite{IaLiPeSq}, the operator $A_p^s \in C(X_p^s(\Omega),X_p^s(\Omega)^\ast)$, where $X_p^s(\Omega)^\ast$ is the dual of $X_p^s(\Omega)$, defined by
\[
A_p^s(u)\, v = \int_{\R^{2N}} \frac{|u(x) - u(y)|^{p-2}\, (u(x) - u(y))\, (v(x) - v(y))}{|x - y|^{N+sp}}\, dx dy, \quad u, v \in X_p^s(\Omega)
\]
satisfies the structural assumptions of \cite[Chapter 1]{MR2640827}. In particular, $A_p^s$ is of type (S), i.e., every sequence $\seq{u_j} \subset X_p^s(\Omega)$ such that
\[
u_j \wto u, \quad A_p^s(u_j)\, (u_j - u) \to 0
\]
has a subsequence that converges strongly to $u$.

\begin{lemma}
The operator $A_p^s$ is strictly monotone, i.e.,
\[
(A_p^s(u) - A_p^s(v))\, (u - v) > 0
\]
for all $u \ne v$ in $X_p^s(\Omega)$.
\end{lemma}

\begin{proof}
By Perera et al.\! \cite[Lemma 6.3]{MR2640827}, it suffices to show that
\[
A_p^s(u)\, v \le \norm{u}^{p-1} \norm{v} \quad \forall u, v \in X_p^s(\Omega)
\]
and the equality holds if and only if $\alpha u = \beta v$ for some $\alpha, \beta \ge 0$, not both zero. We have
\[
A_p^s(u)\, v \le \int_{\R^{2N}} \frac{|u(x) - u(y)|^{p-1}\, |v(x) - v(y)|}{|x - y|^{N+sp}}\, dx dy \le \norm{u}^{p-1} \norm{v}
\]
by the H\"{o}lder inequality. Clearly, equality holds throughout if $\alpha u = \beta v$ for some $\alpha, \beta \ge 0$, not both zero. Conversely, if $A_p^s(u)\, v = \norm{u}^{p-1} \norm{v}$, equality holds in both inequalities. The equality in the second inequality gives
\[
\alpha\, |u(x) - u(y)| = \beta\, |v(x) - v(y)| \quad \text{a.e.\! in } \R^{2N}
\]
for some $\alpha, \beta \ge 0$, not both zero, and then the equality in the first inequality gives
\[
\alpha\, (u(x) - u(y)) = \beta\, (v(x) - v(y)) \quad \text{a.e.\! in } \R^{2N}.
\]
Since $u$ and $v$ vanish a.e.\! in $\R^N \setminus \Omega$, it follows that $\alpha u = \beta v$ a.e.\! in $\Omega$.
\end{proof}

\begin{lemma} \label{Lemma 6}
For each $w \in L^p(\Omega)$, the problem
\begin{equation} \label{16}
\left\{\begin{aligned}
(- \Delta)_p^s\, u & = |w|^{p-2}\, w && \text{in } \Omega\\[10pt]
u & = 0 && \text{in } \R^N \setminus \Omega
\end{aligned}\right.
\end{equation}
has a unique weak solution $u \in X_p^s(\Omega)$. Moreover, the map $J : L^p(\Omega) \to X_p^s(\Omega),\, w \mapsto u$ is continuous.
\end{lemma}

\begin{proof}
The existence follows from a standard minimization argument, and the uniqueness is immediate from the strict monotonicity of the operator $A_p^s$. Let $w_j \to w$ in $L^p(\Omega)$ and let $u_j = J(w_j)$, so
\begin{equation} \label{17}
A_p^s(u_j)\, v = \int_\Omega |w_j|^{p-2}\, w_j\, v\, dx \quad \forall v \in X_p^s(\Omega).
\end{equation}
Testing with $v = u_j$ gives
\[
\norm{u_j}^p = \int_\Omega |w_j|^{p-2}\, w_j\, u_j\, dx \le \pnorm[p]{w_j}^{p-1} \pnorm[p]{u_j}
\]
by the H\"{o}lder inequality, which together with the continuity of the imbedding $X_p^s(\Omega) \hookrightarrow L^p(\Omega)$ shows that $\seq{u_j}$ is bounded in $X_p^s(\Omega)$. So a renamed subsequence of $\seq{u_j}$ converges to some $u$ weakly in $X_p^s(\Omega)$, strongly in $L^p(\Omega)$, and a.e.\! in $\Omega$. An argument similar to that in the proof of Proposition \ref{Proposition 4} shows that $u$ is a weak solution of \eqref{16}, so $u = J(w)$. Testing \eqref{17} with $u_j - u$ gives
\[
A_p^s(u_j)\, (u_j - u) = \int_\Omega |w_j|^{p-2}\, w_j\, (u_j - u)\, dx \to 0,
\]
so $u_j \to u$ for a further subsequence as $A_p^s$ is of type (S).
\end{proof}

\begin{proposition} \label{Proposition 7}
If $\lambda_l < \lambda_{l+1}$, then $\Psi^{\lambda_l}$ has a compact symmetric subset $A_0$ with $i(A_0) = l$.
\end{proposition}

\begin{proof}
Let
\[
\pi_p(u) = \frac{u}{\pnorm[p]{u}}, \quad u \in X_p^s(\Omega) \setminus \set{0}
\]
be the radial projection onto $\M_p = \bgset{u \in X_p^s(\Omega) : \pnorm[p]{u} = 1}$, and let
\[
A = \pi_p(\Psi^{\lambda_l}) = \bgset{w \in \M_p : \norm{w}^p \le \lambda_l}.
\]
Then $i(A) = i(\Psi^{\lambda_l}) = l$ by \ref{i2} of Proposition \ref{Proposition 2} and \eqref{4}. For $w \in A$, let $u = J(w)$, where $J$ is the map defined in Lemma \ref{Lemma 6}, so
\[
A_p^s(u)\, v = \int_\Omega |w|^{p-2}\, wv\, dx \quad \forall v \in X_p^s(\Omega).
\]
Testing with $v = u, w$ and using the H\"{o}lder inequality gives
\[
\norm{u}^p \le \pnorm[p]{w}^{p-1} \pnorm[p]{u} = \pnorm[p]{u}, \qquad 1 = A_p^s(u)\, w \le \norm{u}^{p-1} \norm{w},
\]
so
\[
\norm{\pi_p(u)} = \frac{\norm{u}}{\pnorm[p]{u}} \le \norm{w}
\]
and hence $\pi_p(u) \in A$. Let $\widetilde{J} = \pi_p \comp J$ and let $\widetilde{A} = \widetilde{J}(A) \subset A$. Since the imbedding $X_p^s(\Omega) \hookrightarrow L^p(\Omega)$ is compact and $\widetilde{J}$ is an odd continuous map from $L^p(\Omega)$ to $X_p^s(\Omega)$, then $\widetilde{A}$ is a compact set and $i(\widetilde{A}) = i(A) = l$. Let
\[
\pi(u) = \frac{u}{\norm{u}}, \quad u \in X_p^s(\Omega) \setminus \set{0}
\]
be the radial projection onto $\M$ and let $A_0 = \pi(\widetilde{A})$. Then $A_0 \subset \Psi^{\lambda_l}$ is compact and $i(A_0) = i(\widetilde{A}) = l$.
\end{proof}

We are now ready to prove Theorem \ref{Theorem 1}.

\begin{proof}[Proof of Theorem \ref{Theorem 1}]
We only give the proof of \ref{Theorem 1.ii}. Proof of \ref{Theorem 1.i} is similar and simpler. By Proposition \ref{Proposition 4}, $I_\lambda$ satisfies the \PS{c} condition for all $c < \frac{s}{N}\, S_{s,p}^{N/sp}$, so we apply Theorem \ref{Theorem 3} with $b = \frac{s}{N}\, S_{s,p}^{N/sp}$. By Proposition \ref{Proposition 7}, $\Psi^{\lambda_{k+m}}$ has a compact symmetric subset $A_0$ with
\[
i(A_0) = k + m.
\]
We take $B_0 = \Psi_{\lambda_{k+1}}$, so that
\[
i(S_1 \setminus B_0) = k
\]
by \eqref{4}. Let $R > r > 0$ and let $A$, $B$ and $X$ be as in Theorem \ref{Theorem 3}. For $u \in B_0$,
\[
I_\lambda(ru) \ge \frac{r^p}{p} \left(1 - \frac{\lambda}{\lambda_{k+1}}\right) - \frac{r^{p_s^\ast}}{p_s^\ast\, S_{s,p}^{p_s^\ast/p}}
\]
by \eqref{7}. Since $\lambda < \lambda_{k+1}$ and $p_s^\ast > p$, it follows that $\inf I_\lambda(B) > 0$ if $r$ is sufficiently small. For $u \in A_0 \subset \Psi^{\lambda_{k+1}}$,
\[
I_\lambda(Ru) \le \frac{R^p}{p} \left(1 - \frac{\lambda}{\lambda_{k+1}}\right) - \frac{R^{p_s^\ast}}{p_s^\ast \vol{\Omega}^{s p_s^\ast/N} \lambda_{k+1}^{p_s^\ast/p}}
\]
by the H\"{o}lder inequality, so there exists $R > r$ such that $I_\lambda \le 0$ on $A$. For $u \in X$,
\begin{align*}
I_\lambda(u) & \le \frac{\lambda_{k+1} - \lambda}{p} \int_\Omega |u|^p\, dx - \frac{1}{p_s^\ast \vol{\Omega}^{s p_s^\ast/N}} \left(\int_\Omega |u|^p\, dx\right)^{p_s^\ast/p}\\[10pt]
& \le \sup_{\rho \ge 0}\, \left[\frac{(\lambda_{k+1} - \lambda)\, \rho}{p} - \frac{\rho^{p_s^\ast/p}}{p_s^\ast \vol{\Omega}^{s p_s^\ast/N}}\right]\\[10pt]
& = \frac{s}{N}\, \vol{\Omega} (\lambda_{k+1} - \lambda)^{N/sp}.
\end{align*}
So
\[
\sup I_\lambda(X) \le \frac{s}{N}\, \vol{\Omega} (\lambda_{k+1} - \lambda)^{N/sp} < \frac{s}{N}\, S_{s,p}^{N/sp}
\]
by \eqref{5}. Theorem \ref{Theorem 3} now gives $m$ distinct pairs of (nontrivial) critical points $\pm\, u^\lambda_j,\, j = 1,\dots,m$ of $I_\lambda$ such that
\begin{equation} \label{4.4}
0 < I_\lambda(u^\lambda_j) \le \frac{s}{N}\, \vol{\Omega} (\lambda_{k+1} - \lambda)^{N/sp} \to 0 \text{ as } \lambda \nearrow \lambda_{k+1}.
\end{equation}
Then
\[
|u^\lambda_j|_{p_s^\ast}^{p_s^\ast} = \frac{N}{s} \left[I_\lambda(u^\lambda_j) - \frac{1}{p}\, I_\lambda'(u^\lambda_j)\, u^\lambda_j\right] = \frac{N}{s}\, I_\lambda(u^\lambda_j) \to 0
\]
and hence $u^\lambda_j \to 0$ in $L^p(\Omega)$ also by the H\"{o}lder inequality, so
\[
\|u^\lambda_j\|^p = p\, I_\lambda(u^\lambda_j) + \lambda\, |u^\lambda_j|_p^p + \frac{p}{p_s^\ast}\, |u^\lambda_j|_{p_s^\ast}^{p_s^\ast} \to 0. \QED
\]
\end{proof}

\def\cdprime{$''$}

\end{document}